\documentclass[12pt]{amsart}
\usepackage{amsmath}
\usepackage{amsxtra}
\usepackage{enumitem}
\usepackage{amscd}
\usepackage{tikz-cd}
\usepackage{amsthm}
\usepackage{hyperref}
\usepackage[T1]{fontenc}
\usepackage{amsfonts}
\usepackage{amssymb}
\usepackage{eucal}
\usepackage[all]{xypic}
\usepackage{color} 
\newtheorem{Theo}[subsubsection]{Theorem}
\newtheorem{Theor}{Theorem}
\newtheorem{Theore}{Theorem}

\newtheorem{cor}[subsubsection]{Corollary}
\newtheorem{lem}[subsubsection]{Lemma}

\newtheorem{prop}[subsubsection]{Proposition}

\theoremstyle{definition}

\newtheorem{fact}[subsubsection]{Fact} 
\newtheorem{rem}[subsubsection]{Remark}

\newtheorem{Convention}[subsubsection]{Convention}
\newtheorem{question}[subsubsection]{Question}
\newtheorem{Obs}[subsubsection]{Observation}
\newtheorem{example}[subsubsection]{Example}

\newtheorem{definition}[subsubsection]{Definition}

\newtheorem{ex}[subsection]{Exercise}

\usepackage{yhmath}
\DeclareSymbolFont{largesymbols}{OMX}{yhex}{m}{n}
\DeclareMathAccent{\widetilde}{\mathord}{largesymbols}{"65}

\newcommand{\bp}{\begin{prop}}
\newcommand{\ep}{\end{prop}}

\newcommand{\bl}{\begin{lem}}
\newcommand{\el}{\end{lem}}
\newcommand{\bex}{\begin{ex} \rm}
\newcommand{\eex}{\end{ex}}
\newcommand{\bt}{\begin{Theo}}
\newcommand{\et}{\end{Theo}}
\newcommand{\bq}{\begin{question}}
\newcommand{\eq}{\end{question}}

\newcommand{\bc}{\begin{cor}}
\newcommand{\ec}{\end{cor}}
\newcommand{\bob}{\begin{Obs}}
\newcommand{\eob}{\end{Obs}}

\newcommand{\nc}{\newcommand}
\nc{\renc}{\renewcommand}
\nc{\ssec}{\subsection}
\nc{\sssec}{\subsubsection} 
\nc\ol{\overline}
\nc\wt{\widetilde}
\nc\wh{\widehat}
\nc\tboxtimes{\wt{\boxtimes}}

\renc{\d}{{\delta}}
\nc{\Aa}{{\mathbb{A}}}
\nc{\Bb}{{\mathbb{B}}}
 \nc{\Gg}{{\mathbb{G}}}  
\nc{\Hh}{{\mathbb{H}}}
 \nc{\Nn}{{\mathbb{N}}}
\nc{\Pp}{{\mathbb{P}}}
\nc{\Rr}{{\mathbb{R}}}
\newcommand{\F}{\mathbb{F}}
\nc{\BV}{{\mathbb{V}}}
\nc{\BW}{{\mathbb{W}}}
\newcommand{\Z}{\mathbb{Z}}
\newcommand{\N}{\mathbb{N}}

\nc{\Qq}{{\mathbb{Q}}}
\nc{\Ss}{{\mathbb{S}}}
\nc{\Cc}{{\mathbb{C}}}
\nc{\Ff}{{\mathbb{F}}}
 \nc{\EL}{{L_\infty}}

\nc{\CA}{{\mathcal{A}}}
\nc{\CB}{{\mathcal{B}}}

\nc{\CE}{{\mathcal{E}}}
\nc{\CF}{{\mathcal{F}}}

\nc{\Las}{\mathsf{Las}}
\nc{\CG}{{\mathcal{G}}}

\nc{\CL}{{\mathcal{L}}}
\nc{\CC}{{\mathcal{C}}}
\nc{\CM}{{\mathcal{M}}}

\nc{\CN}{{\mathcal{N}}}
\nc{\Oog}{{\mathbb{O}}}
\nc{\Oo}{{\mathcal{O}}}
\nc{\CP}{{\mathcal{P}}}
\nc{\CQ}{{\mathcal{Q}}}
\nc{\CR}{{\mathcal{R}}}
\nc{\CS}{{\mathcal{S}}}
\nc{\CT}{{\mathcal{T}}}
\nc{\CU}{{\mathcal{P}}}

\nc{\CV}{{\mathcal{V}}}
 
\nc{\CW}{{\mathcal{W}}}
\nc{\CZ}{{\mathcal{Z}}}

\nc{\cM}{{\check{\mathcal M}}{}}
\nc{\csM}{{\check{\mathcal A}}{}}
\nc{\oM}{{\overset{\circ}{\mathcal M}}{}}
\nc{\obM}{{\overset{\circ}{\mathbf M}}{}}
\nc{\oCA}{{\overset{\circ}{\mathcal A}}{}}
\nc{\obA}{{\overset{\circ}{\mathbf A}}{}}
\nc{\ooM}{{\overset{\circ}{M}}{}}
\nc{\osM}{{\overset{\circ}{\mathsf M}}{}}
\nc{\vM}{{\overset{\bullet}{\mathcal M}}{}}
\nc{\nM}{{\underset{\bullet}{\mathcal M}}{}}
\nc{\oD}{{\overset{\circ}{\mathcal D}}{}}
\nc{\obD}{{\overset{\circ}{\mathbf D}}{}}
\nc{\oA}{{\overset{\circ}{\mathbb A}}{}}
\nc{\op}{{\overset{\bullet}{\mathbf p}}{}}
\nc{\cp}{{\overset{\circ}{\mathbf p}}{}}
\nc{\oU}{{\overset{\bullet}{\mathcal U}}{}}
\nc{\oZ}{{\overset{\circ}{\mathcal Z}}{}}
\nc{\ofZ}{{\overset{\circ}{\mathfrak Z}}{}}
\nc{\oF}{{\overset{\circ}{\fF}}}

\nc{\fa}{{\mathfrak{a}}}
\nc{\fb}{{\mathfrak{b}}}
\nc{\fg}{{\mathfrak{g}}}
\nc{\fgt}{{\fg}_!}
\nc{\fgl}{{\mathfrak{gl}}}
\nc{\fh}{{\mathfrak{h}}}
\nc{\fj}{{\mathfrak{j}}}
\nc{\fm}{{\mathfrak{m}}}
\nc{\ft}{{\mathfrak{t}}}
\nc{\fn}{{\mathfrak{n}}}
\nc{\fu}{{\mathfrak{u}}}
\nc{\fp}{{\mathfrak{p}}}
\nc{\fr}{{\mathfrak{r}}}
\nc{\fs}{{\mathfrak{s}}}
\nc{\fsl}{{\mathfrak{sl}}}
\nc{\hsl}{{\widehat{\mathfrak{sl}}}}
\nc{\hgl}{{\widehat{\mathfrak{gl}}}}
\nc{\hg}{{\widehat{\mathfrak{g}}}}
\nc{\chg}{{\widehat{\mathfrak{g}}}{}^\vee}
\nc{\hn}{{\widehat{\mathfrak{n}}}}
\nc{\chn}{{\widehat{\mathfrak{n}}}{}^\vee}

\nc{\fA}{{\mathfrak{A}}}
\nc{\fB}{{\mathfrak{B}}}
\nc{\fD}{{\mathfrak{D}}}
\nc{\fE}{{\mathfrak{E}}}
\nc{\fF}{{\mathfrak{F}}}
\nc{\fG}{{\mathfrak{G}}}
\nc{\fK}{{\mathfrak{K}}}
\nc{\fL}{{\mathfrak{L}}}
\nc{\fM}{{\mathfrak{M}}}
\nc{\fN}{{\mathfrak{N}}}
\nc{\fP}{{\mathfrak{P}}}
\nc{\fU}{{\mathfrak{U}}}
\nc{\fV}{{\mathfrak{V}}}
\nc{\fZ}{{\mathfrak{Z}}}
\newcommand{\Q}{\mathbb{Q}}
\nc{\bb}{{\mathbf{b}}}
\nc{\bd}{\partial}
\nc{\be}{{\mathbf{e}}}
\nc{\bj}{{\mathbf{j}}}
\nc{\bn}{{\mathbf{n}}}
\nc{\bF}{{\mathbf{F}}}
\nc{\bu}{{\mathbf{u}}}
\nc{\bv}{{\mathbf{v}}}
\nc{\bx}{{\mathbf{x}}}
\nc{\bs}{{\mathbf{s}}}
\nc{\by}{{\bar{y}}}
\nc{\bw}{{\mathbf{w}}}
\nc{\bA}{{\mathbf{A}}}
\nc{\bK}{{\mathbf{K}}}
\nc{\bI}{{\mathbf{I}}}
\nc{\bB}{{\mathbf{B}}}
\nc{\bG}{{\mathbf{G}}}

\nc{\bD}{{\mathbf{D}}}
\nc{\bP}{{\mathbf{P}}}
\nc{\bH}{{\mathbf{H}}}
\nc{\bM}{{\mathbf{M}}}
\nc{\bN}{{\mathbf{N}}}
\nc{\bV}{{\mathbf{V}}}
\nc{\bU}{{\mathbf{U}}}
\nc{\bL}{{\mathbf{L}}}

\nc{\bW}{{\mathbf{W}}}
\nc{\bX}{{\mathbf{X}}}
\nc{\bY}{{\mathbf{Y}}}
\nc{\bZ}{{\mathbf{Z}}}
\nc{\bS}{{\mathbf{S}}}
\nc{\bSi}{{\bar{\Sigma}}}
\nc{\sA}{{\mathsf{A}}}
\nc{\sB}{{\mathsf{B}}}
\nc{\sC}{{\mathsf{C}}}
\nc{\sD}{{\mathsf{D}}}
\nc{\sF}{{\mathsf{F}}}
\nc{\sG}{{\mathsf{G}}}
\nc{\sK}{{\mathsf{K}}}
\nc{\sM}{{\mathsf{M}}}
\nc{\sO}{{\mathsf{O}}}
\nc{\sQ}{{\mathsf{Q}}}
\nc{\sP}{{\mathsf{P}}}
\nc{\sZ}{{\mathsf{Z}}}
\nc{\sfp}{{\mathsf{p}}}
\nc{\sr}{{\mathsf{r}}}
\nc{\sg}{{\mathsf{g}}}
\nc{\sff}{{\mathsf{f}}}
\nc{\sfb}{{\mathsf{b}}}
\nc{\sfc}{{\mathsf{c}}}
\nc{\sd}{{\ltimes}} 

\nc{\tH}{{\widetilde{H}}}
\nc{\tA}{{\widetilde{\mathbf{A}}}}
\nc{\tB}{{\widetilde{\mathcal{B}}}}
\nc{\tg}{{\widetilde{\mathfrak{g}}}}
\nc{\tG}{{\widetilde{G}}}

\nc{\TM}{{\widetilde{\mathbb{M}}}{}}
\nc{\tO}{{\widetilde{\mathsf{O}}}{}}
\nc{\tU}{\widetilde{U}}
\nc{\TZ}{{\tilde{Z}}}
\nc{\tx}{{\tilde{x}}}
\nc{\tq}{{\tilde{q}}}

\nc{\tfP}{{\widetilde{\mathfrak{P}}}{}}
\nc{\tz}{{\tilde{\zeta}}}
\nc{\tmu}{{\tilde{\mu}}}

  \nc{\vol}{{\mathop{\operatorname{\rm vol\,}}}}
  
  \nc{\gal}{{\mathop{\operatorname{\rm Gal\,}}}}
  \nc{\cl}{{\mathop{\operatorname{\rm cl}}}}
  \nc{\disc}{{\mathop{\operatorname{\rm disc}}}}
  \nc{\Sym}{{\mathop{\operatorname{\rm Sym}}}}
   \nc{\Aut}{{\mathop{\operatorname{\rm Aut}}}}
 \nc{\Spec}{{\mathop{\operatorname{\rm Spec}}}}
  \nc{\spec}{{\mathop{\operatorname{\rm Spec}}}}
\nc{\Ker}{{\mathop{\operatorname{\rm Ker}}}}
 \nc{\dom}{{\mathop{\operatorname{\rm dom}}}}
\nc{\End}{{\mathop{\operatorname{\rm End}}}}
 \nc{\Hom}{\operatorname{\Hom}}
 \nc{\GL}{{\mathop{\operatorname{\rm GL}}}}
 \nc{\Id}{{\mathop{\operatorname{\rm Id}}}}
 \nc{\rk}{{\mathop{\operatorname{\rm rk}}}}
 \nc{\length}{{\mathop{\operatorname{\rm length}}}}
\nc{\supp}{{\mathop{\operatorname{\rm supp} \, }}}
\nc{\val}{{\rm val}}
\nc{\res}{{\mathop{\operatorname{\rm res}}}}

\def\Ind#1#2#3{{#1} {\downarrow}_{#3} {#2} }

\nc{\seq}[1]{\stackrel{#1}{\sim}}

\def\beq#1{\begin{equation} \label{ #1}}
\def\eeq{\end{equation}}

\def\prf{\begin{proof}}
\def\pv{\end{proof} }
 \def\eprf{\end{proof} }

 \renc{\b}{{\beta}}

\def\Ind#1#2{#1\setbox0=\hbox{$#1x$}\kern\wd0\hbox to 0pt{\hss$#1\mid$\hss}
\lower.9\ht0\hbox to 0pt{\hss$#1\smile$\hss}\kern\wd0}

\usepackage{fancyhdr}
\usepackage{color}
\usepackage{hyperref}

\usepackage{url}

 \makeatletter
\def\@setthanks{\vspace{-\baselineskip}\def\thanks##1{\@par##1\@addpunct.}\thankses}
\makeatother
\title[Undecidability of the asymptotic theory of $p$-adic fields]{An undecidability result for the asymptotic theory of $p$-adic fields}
\author{Konstantinos Kartas}
\thanks{During this research, the author was funded by EPSRC grant EP/20998761 and was also supported by the Onassis Foundation - Scholarship ID: F ZP 020-1/2019-2020.}

\newcommand{\Addresses}{{
  \bigskip
  \footnotesize

\textsc{Mathematical Institute, Woodstock Road, Oxford OX2 6GG.}\par\nopagebreak
  \textit{E-mail address}: \texttt{kartas@maths.ox.ac.uk}
}}

\begin{document}
\maketitle

\begin{abstract}
Fix a prime $p$. We prove that the set of sentences true in all but finitely many finite extensions of $\Q_p$ is undecidable in the language of valued fields with a cross-section. The proof goes via reduction to positive characteristic, ultimately adapting Pheidas' proof of the undecidability of $\F_p(\!(t)\!)$ with a cross-section. This answers a variant of a question of Derakhshan-Macintyre. 
\end{abstract}
\setcounter{tocdepth}{1}
\tableofcontents

\section*{Introduction}
Let $L_{\text{val},\times}$ be the language of valued fields with a cross-section (see \S \ref{AKElanguage}). Ax-Kochen \cite{AK} and independently Ershov \cite{Ershov} showed that $\Q_p$, equipped with the normalized cross section $s:n\mapsto p^n$, is decidable in $L_{\text{val},\times}$. Every finite extension of $\Q_p$ is also decidable in $L_{\text{val},\times}$, for a suitable choice of a cross-section (see \S 2.4 \cite{ThesisKartas}). Combining the results of \cite{AK} with the theory of pseudofinite fields, Ax showed in Theorem 17 \cite{Ax} that the (asymptotic) theory of $\{\Q_p:p\in \mathbb{P}\}$ is decidable in $L_{\text{val},\times}$ and also that the (asymptotic) theory of $\{ \Q_p(\zeta_n):p\nmid n\}$, namely the collection of all finite \textit{unramified} extensions of $\Q_p$, is decidable in $L_{\text{val},\times}$. Recall that the asymptotic theory of a class $\mathcal{C}$ of $L$-structures is defined to be the set of sentences $\phi \in L$ which are true in all but finitely many $M\in \mathcal{C}$. 

In sharp contrast, we have that $\F_p(\!(t)\!)$ is \textit{undecidable} in $L_{\text{val},\times}$. This was already known to Ax (unpublished) and an elementary proof was given later by Becker-Denef-Lipshitz \cite{BDL}, which was also reworked by Cherlin in \S 4 \cite{Cherlin}. Pheidas \cite{Pheid} generalized the result for $k(\!(t)\!)$, where $k$ is an \textit{arbitrary} field of characteristic $p$, and showed that already the \textit{existential} $L_{\text{val},\times}$-theory is undecidable. The decidability problem for $\F_p(\!(t)\!)$ in the language of rings $L_{\text{rings}}$ is still an open problem.

An important related open question in mixed characteristic is whether $\text{Th}(\{K:[K:\Q_p]<\infty\})$, i.e. the theory of \textit{all} finite $p$-adic extensions, is decidable. This question was first raised (in print) by Derakhshan-Macintyre in \S 9 \cite{DerMac} for the language of rings $L_{\text{rings}}$. In Theorem 9.1 \cite{DerMac}, they showed that the theory of adele rings of \textit{all} number fields in $L_{\text{rings}}$ is decidable if and only if $\text{Th}(\{K:[K:\Q_p]<\infty\})$ is decidable in $L_{\text{rings}}$, for each prime $p$. 
It is also natural to consider $\text{Th}(\{K:[K:\Q_p]<\infty\})$ in any of the standard expansions $L$ of $L_{\text{rings}}$ (see e.g., pg. 20-22 \cite{Der}), and ask whether it is decidable (Problem 6.2 \cite{Der}).\\

In the present paper, we resolve negatively the asymptotic version of this problem in the presence of a cross-section. While each \textit{individual} finite $p$-adic extension is decidable in $L_{\text{val},\times}$, the asymptotic theory of all of them is less well-behaved:
\begin{Theore} \label{main}
The asymptotic $L_{\text{val},\times}$-theory of $\{K:[K:\Q_p]<\infty\}$ is undecidable.
\end{Theore}
We note that for each $K$ with $[K:\Q_p]< \infty$ there are many choices of a cross-section and \textit{any} such choice will lead to an undecidability result (see Convention \ref{convent}). By carefully keeping track of quantifiers, we will prove in \S \ref{mainsec} that already the asymptotic $\exists \forall $-theory of $\{K:[K:\Q_p]<\infty\}$ is undecidable in $L_{\text{val},\times}$. 

The key idea of the proof will be to use highly ramified $p$-adic fields to approximate $\F_p(\!(t)\!)$  \`a la Krasner-Kazhdan-Deligne (see Section \ref{localfieldapprox}) and then adapt Pheidas' proof of the undecidability of $\F_p(\!(t)\!)$ in $L_{\text{val},\times}$ (see \cite{Pheid}). In more detail, the proof of Theorem \ref{main} consists of the following steps:
\begin{enumerate}
\item Encode the asymptotic theory of totally ramified $p$-adic fields, i.e., the asymptotic theory of $\{K:[K:\Q_p]=e(K/\Q_p)<\infty\}$ (see the proof of Theorem \ref{mainagain}).

\item Observe that $\Oo_K/(p)\cong \F_p[t]/(t^e)$, for $K/\Q_p$ totally ramified of degree $e$, and thereby encode the asymptotic theory of $\{\F_p[t]/(t^n):n\in \N\}$ in $L_t$ with a predicate for powers of $t$ (see Corollary \ref{0topinter}).

\item Encode the asymptotic theory of \textit{truncated fragments} of $(\N;0,1,+,|_p)$ (see \S \ref{truncatedversions}), where $m\mid_p n$ if and only if $n=p^s\cdot m$ for some $s\in \N$.  

\item Show that the latter is undecidable by encoding the Diophantine problem of $(\N;0,1,+,|_p)$ (see \S \ref{undecoffragments}).

\end{enumerate}

We note that Derakhshan-Macintyre had already suggested that the difficulty in understanding the (asymptotic) theory of $p$-adic fields must lie in unbounded ramification (see the last paragraph of \cite{DerMac}). This is precisely the ingredient that makes the transition to positive characteristic work and is the reason why this asymptotic class exhibits a different behavior from the two asymptotic theories mentioned in the first paragraph of the introduction.
\section{Preliminaries}

\subsection{Interpretability} \label{intersec}
Our formalism follows closely \S 5.3 \cite{Hod}, where details and proofs may be found. 

\subsubsection{Interpretations of structures}
Given a language $L$, an \textit{unnested} atomic $L$-formula is one of the form $x=y$ or $x=c$ or $F(\overline{x})=y$ or $R\overline{x}$, where $x,y$ are variables, $c$ is a constant symbol, $F$ is a function symbol and $R$ is a relation symbol of the language $L$.
\begin{definition} \label{interdef}
An $n$-dimensional interpretation of an $L$-structure $M$ in the $L'$-structure $N$ is a triple consisting of:
\begin{enumerate}
\item An $L'$-formula $\partial_{\Gamma}(x_1,...,x_n)$.
\item A map $\phi\mapsto \phi_{\Gamma}$, that takes an unnested atomic $L$-formula $\phi(x_1,...,x_m)$ and sends it to an $L'$-formula $\phi_{\Gamma}(\overline{y}_1,...,\overline{y}_m)$, where each $\overline{y}_i$ is an $n$-tuple of variables.

\item A surjective map $f_{\Gamma}:\partial_{\Gamma}(N^n)\twoheadrightarrow M$.
\end{enumerate}
such that for all unnested atomic $L$-formulas $\phi(x_1,...,x_m)$ and all $\overline{a}_i\in \partial_{\Gamma}(N^n)$, we have 
$$M\models \phi(f_{\Gamma} (\overline{a}_1), ...,f_{\Gamma} (\overline{a}_m)) \iff N \models  \phi_{\Gamma}(\overline{a}_1,...,\overline{a}_m)  $$
\end{definition}
An \textit{interpretation} of an $L$-structure $M$ in the $L'$-structure $N$ is an $n$-dimensional interpretation, for some $n\in \N$. In that case, we also say  that $M$ is \textit{interpretable} in $N$. The formulas $\partial_{\Gamma}$ and $\phi_{\Gamma}$ (for all unnested atomic $\phi$) are the \textit{defining formulas} of $\Gamma$. 
%
%

\bp [Reduction Theorem 5.3.2 \cite{Hod}]\label{prophod}
Let $\Gamma$ be an $n$-dimensional interpretation of an $L$-structure $M$ in the $L'$-structure $N$. There exists a map $\phi\mapsto \phi_{\Gamma}$, extending the map of Definition \ref{interdef}$(2)$, such that for every $L$-formula $\phi(x_1,...,x_m)$ and all $\overline{a}_i\in \partial_{\Gamma}(N^n)$, we have that
$$M\models \phi(f_{\Gamma} (\overline{a}_1), ...,f_{\Gamma} (\overline{a}_m)) \iff N \models  \phi_{\Gamma}(\overline{a}_1,...,\overline{a}_m)  $$
\ep 
\begin{proof}
We describe how $\phi \mapsto \phi_{\Gamma}$ is built, for completeness (omitting details). By Corollary 2.6.2 \cite{Hod}, every $L$-formula is equivalent to one in which all atomic subformulas are unnested. One can then construct $\phi \mapsto \phi_{\Gamma}$ by induction on the complexity of formulas. The base case is handled by Definition \ref{interdef}$(2)$. This definition extends inductively according to the following rules:
\begin{enumerate}
\item $(\lnot \phi)_{\Gamma}=\lnot(\phi)_{\Gamma}$.
\item $(\bigwedge_{i=1}^n \phi_i)_{\Gamma}=\bigwedge (\phi_i)_{\Gamma}$.
\item $(\forall \phi)_{\Gamma}= \forall x_1,...,x_n (\partial_{\Gamma} (x_1,...,x_n) \rightarrow \phi_{\Gamma})$
\item $(\exists \phi)_{\Gamma}= \exists x_1,...,x_n (\partial_{\Gamma} (x_1,...,x_n) \land \phi_{\Gamma})$
\end{enumerate}
The resulting map satisfies the desired conditions of the Proposition.
\end{proof}
\begin{definition}
The map $\text{Form}_L\to \text{Form}_{L'}:\phi\mapsto \phi_{\Gamma}$ constructed in the proof of Proposition \ref{prophod} is called the \textit{reduction} map of the interpretation $\Gamma$.
\end{definition}
\subsubsection{Complexity of interpretations}
The complexity of the defining formulas of an interpretation defines a measure of complexity of the interpretation itself:

\begin{definition} [\S 5.4$(a)$ \cite{Hod}]
An interpretation $\Gamma$ of an $L$-structure $M$ in an $L'$-structure $N$ is quantifier-free if the defining formulas of $\Gamma$ are quantifier-free. Other syntactic variants are defined analogously (e.g., existential interpretation).
\end{definition} 
\begin{rem}
Note that the reduction map of an existential interpretation sends \textit{positive existential} formulas to existential formulas but does \textit{not} necessarily send existential formulas to existential formulas. 
\end{rem}

\subsubsection{Recursive interpretations}

\begin{definition} [Remark 4, pg. 215 \cite{Hod} ]
Suppose $L$ is a recursive language. Let $\Gamma$ be an interpretation of an $L$-structure $M$ in the $L'$-structure $N$. We say that the interpretation $\Gamma$ is \textit{recursive} if the the map $\phi \mapsto \phi_{\Gamma}$ on unnested atomic formulas is recursive.
\end{definition}

\begin{rem} [Remark 4, pg. 215 \cite{Hod}]
If $\Gamma$ is a recursive interpretation of an $L$-structure $M$ in the $L'$-structure $N$, then the reduction map of $\Gamma$ is also recursive.
\end{rem}

\subsection{Decidability}
\subsubsection{Definition}
Fix a \textit{countable} language $L$ and let $\text{Sent}_L$ be the set of well-formed $L$-sentences, identified with $\N$ via some G\"odel numbering. 
Let $T$ be an $L$-theory, not necessarily complete. Recall that $T$ is decidable if we have an algorithm to decide whether $T \models \phi$, for any given $\phi \in \text{Sent}_L$. More formally, let $\chi_{T}:\text{Sent}_L \to \{0,1\}$ be the \text{partial} characteristic function of $T\subseteq \text{Sent}_L$. We say that $T$ is \textit{decidable} if $\chi_T$ is a partial recursive function. Let $M$ be an $L$-structure. We say that $M$ is decidable if $\text{Th}(M)$ is decidable. 
\subsubsection{Turing reducibility of theories}
See Definition 14.3 \cite{Papa} for the formal definition of a Turing machine with an oracle.  
\begin{definition}
A theory $T$ is \textit{Turing reducible} to a theory $T'$ if there is a Turing machine which decides membership in $T$ using an \textit{oracle} for $T'$. 
\end{definition}
\begin{rem}
In particular, if $T$ is Turing reducible to $T'$ and $T'$ is decidable, then so is $T$. 
\end{rem}

\begin{example}
If $\Gamma$ is a recursive interpretation of an $L$-structure $M$ in the $L'$-structure $N$, then $\text{Th}(M)$ is Turing reducible to $\text{Th}(N)$. Indeed, for any given $\phi\in \text{Sent}_L$, we have that $M\models \phi \iff N\models \phi_{\Gamma}$, where $\phi \mapsto \phi_{\Gamma}$ is the reduction map of $\Gamma$. This furnish us with an algorithm to decide whether $M\models \phi$ using an oracle for $\text{Th}(N)$. 
\end{example}

\subsection{Languages}
We write $L_{\text{oag}}=\{0,+,<\}$ for the language of ordered abelian groups and $L_{\text{rings}}=\{0,1,+,\cdot\}$ for the language of rings.  
\subsubsection{Valued fields language}
Let $L_{\text{val}}$ be the three-sorted language of valued fields, with sorts for the field, the value group and the residue field. 
\begin{itemize}
\item The field sort $\mathbf{K}$ is equipped with the language of rings $L_{\text{rings}}$.

\item The value group sort $\mathbf{\Gamma}$ is equipped with $L_{\text{oag}}$, together with a constant symbol for $\infty$.
\item The residue field sort $\mathbf{k}$ is equipped the language of rings $L_{\text{rings}}$.
\end{itemize}
We also have function symbols for the valuation map $v:\mathbf{K} \to \mathbf{\Gamma}$ and a residue map $\text{res}: \mathbf{K}\to \mathbf{k}$ (where $\text{res}(x)=0$ when $vx<0$ by convention). 
\begin{Convention} \label{convent1}
We shall write $x\in \mathcal{O}$ as an abbreviation of the formula $x\in \mathbf{K} \land vx\geq 0$.
\end{Convention}

\subsubsection{Ax-Kochen/Ershov language} \label{AKElanguage}
Historically, the Ax-Kochen/Ershov formalism also included a function symbol for a cross-section, i.e. group homomorphism $s:\Gamma \to K^{\times}$ satisfying $v\circ s=\text{id}_{\Gamma}$ (see \cite{AK} and \cite{Kochen}). One can also extend $s$ by defining $s(\infty)=0$. We write $L_{\text{val},\times}$ for the language $L_{\text{val}}$ enriched with such a cross-section symbol $s: \mathbf{\Gamma} \to \mathbf{K}$. 

\subsection{Ax-Kochen/Ershov Theorem}
Among other results, Ax-Kochen \cite{AK} and independently Ershov \cite{Ershov} obtained the following:
\begin{fact} [Ax-Kochen/Ershov]
The field $\Q_p$, equipped with the normalized cross-section $s:n\mapsto p^n$, is decidable in $L_{\text{val},\times}$.
\end{fact}

\begin{rem} \label{akeremcross}
$(a)$ More generally, by Theorem 4 \cite{Kochen}, any unramified henselian field $(K,v)$, equipped with a \textit{normalized} cross-section (viz., $s(1)=p$) and with perfect residue field  $k$, is decidable in $L_{\text{val},\times}$. \\
$(b)$ Finite extensions of $\Q_p$ are also decidable in $L_{\text{val},\times}$, for a suitable choice of a cross-section (see \S 2.4 \cite{ThesisKartas}).
\end{rem}
Theorem 12 \cite{AK3} in fact shows that $\Q_p$ admits quantifier elimination in $L_{\text{val},\times}$ relative to the value group. The definable sets of the latter are perfectly understood by classical quantifier elimination results for Presburger arithmetic. Nevertheless, the following is worth noting:
\begin{rem} \label{macintyrerem}
The definable sets in the Ax-Kochen/Ershov language are complicated and are generally not definable in the valued field language. For instance, the image of the cross-section is not definable without the cross-section (see Example, pg. 609 \cite{Mac}). For this reason, at least for the purpose of studying definable subsets of the $p$-adics, the Ax-Kochen/Ershov formalism was superseded by Macintyre's language (see pg. 606 \cite{Mac}). 
\end{rem}
Despite Remark \ref{macintyrerem}, the Ax-Kochen/Ershov formalism has remained relevant. The fact that $p$-adic fields are decidable in such an expressive language is a strong result. As was mentioned in the introduction, this is in stark contrast with the fact that positive characteristic local fields are undecidable in $L_{\text{val},\times}$. 

%
%
%

\section{Local field approximation} \label{localfieldapprox}
\subsection{Motivation} \label{kkd}
A powerful philosophy, often referred to as the Krasner-Kazhdan-Deligne principle (due to \cite{Krasner}, \cite{Kazhdan} and \cite{Deligne}), says that a highly ramified $p$-adic field $K$ (e.g., $K=\Q_p(p^{1/n})$ with $n$ large) is in many respects "close" to a positive characteristic valued field. Although this is reflected in many aspects of $K$, perhaps the most elementary one is that the residue ring $\Oo_K/(p)$ "approximates" a positive characteristic valuation ring (see Lemma \ref{modp}, Remark \ref{kazrem}). Our goal in this section will be to prove Corollary \ref{0topinter}, which will serve as a bridge between Pheidas' work in positive characteristic and our problem in mixed characteristic (see \S \ref{motiv}). 
\subsection{Computations modulo $p$}

Let $K/\Q_p$ be a finite \textit{totally ramified} extension of degree $n$, with value group $\Gamma$ and residue field $k$. 
For completeness, we record here the following computation:
\bl \label{modp}
For each uniformizer $\pi$ of $\Oo_K$ we have an isomorphism $\Oo_K/(p)\cong \F_{p}[t]/(t^n)$, which maps the image of $\pi$ in $\Oo_K/(p)$ to the image of $t$ in $\F_{p}[t]/(t^n)$.
\el 
\begin{proof}
Write $\Oo_K=\Z_p[\pi]$, where $\pi$ is a root of an Eisenstein polynomial $E(t)=t^n+a_{n-1}t^{n-1}+....+a_0 \in \Z_p[t]$ (see Proposition 11, pg.52 \cite{Lang}). In particular, we have that $a_i\equiv 0 \mod p \Z_p$ and the reduction of $E(t)$ modulo $p\Z_p$ is equal to $\overline{E}(t)=t^n \in \F_{p}[t]$. We now compute 
$$\Oo_K/(p)=\Z_p[t]/(p,E(t))\cong \F_{p}[t]/(\overline{E}(t))=\F_{p}[t]/(t^n)$$
and observe that the above isomorphism sends $\pi+(p)$ to $t+(t^n)$.
\end{proof}
\begin{rem} \label{kazrem}
If $K,\pi$ are as above, then $\Oo_K/(\pi^n)=\Oo_K/(p) \cong \F_{p}[t]/(t^n)\cong \F_{p}[\![t]\!]/(t^n)$ and Kazhdan says that $K$ is $n$-close to $\F_{p}[\![t]\!]$ (see \S 0 \cite{Kazhdan}). In what follows, it will indeed be useful to think of $\Oo_K/(p)$ as being very close to $\F_{p}[\![t]\!]$ for large $n$ (see \ref{planofaction}).
\end{rem}

\subsubsection{Residue rings}
For each $n\in \N$, we view $\F_p[t]/(t^n)$ as an $L_t\cup P$-structure, where $L_t$ is the language of rings $L_{\text{rings}}$ together with a constant symbol for $t$ and $P$ is a unary predicate, whose interpretation is the set $\{0,1,t,...,t^{n-1}\}$. Note that we have tacitly replaced the equivalence class $t^k+(t^n)$ with $t^k$, which is harmless and common when dealing with truncated/modular arithmetic.  
\subsubsection{Interpreting $\F_p[t]/(t^n)$}
As a consequence of Lemma \ref{modp}, we obtain:
\bp \label{intercor}
Let $K/\Q_p$ be totally ramified of degree $n\in \N$ and $s:\Gamma \to K^{\times}$ be a cross-section. The structure $\F_p[t]/(t^n)$ in $L_t\cup P$ is $\exists \forall$-interpretable in $K$ in the language $L_{\text{val},\times}$. Moreover, the reduction map $\phi\mapsto \phi_{\Delta_K}$ does not depend on $K$ or the choice of $s$.
\ep 
\begin{proof}
Let $\gamma$ be the minimal positive element in $\Gamma$.  It is definable by the $\forall$-formula in the free variable $\gamma \in \mathbf{\Gamma}$ written below
$$\gamma >0\land \forall \delta \in \mathbf{\Gamma}^{>0} (\gamma \leq \delta) $$ 
henceforth abbreviated by $(\gamma \mbox{ minimal positive})$.

We now define a $1$-dimensional interpretation $\Delta_K$ of $\F_p[t]/(t^n)$ in $K$. Take $\partial_{\Delta_K}(x)$ to be the formula $x\in \mathcal{O}$. The reduction map on unnested atomic formulas is described as follows:
\begin{enumerate}
\item If $\phi(x)$ is the formula $x=0$ (resp. $x=1$ and $x=t$), we take $\phi_{\Delta_K}(x)$ to be the formula $\exists y\in \mathcal{O} (x=py)$ (resp. $\exists y\in \mathcal{O} (x=1+py)$ and $\exists \gamma \in \mathbf{\Gamma}[(\gamma \mbox{ minimal positive})\land (x=s(\gamma))]$). 

\item If $\phi(x,y)$ is the formula $x=y$, we take $\phi_{\Delta_K}(x,y)$ to be the formula $\exists z \in \mathcal{O}(x=y+pz)$.
\item If $\phi(x,y,z)$ is $x\diamond y=z$, then we take $\phi_{\Delta_K}$ to be the formula $\exists w \in \mathcal{O} (x\diamond y=z +pw)$, where $\diamond$ is either $\cdot$ or $+$.

\item If $\phi(x)$ is the formula $x\in P$, we take $\phi_{\Delta_K}(x)$ to be the formula $\exists \gamma \in \mathbf{\Gamma}^{\geq 0} (x=s(\gamma))$
\end{enumerate}
The coordinate map $f_{\Delta_K}: \Oo_K\to \F_p[t]/(t^n)$ is the projection modulo $p$, given by Lemma \ref{modp}. The isomorphism $\Oo_K/(p)\cong \F_p[t]/(t^n)$ identifies $\pi$ with $t$ and the image of $s(\Gamma^{\geq 0})$ in $\Oo_K/(p)$ is equal to $\{0,1,t,...,t^{n-1}\}$, via the above identification. One readily checks that the above data defines an $\exists \forall$-interpretation of the $L_t\cup P$-structure $\F_p[t]/(t^n)$ in the $L_{\text{val},\times}$-structure $K$. 

Finally, the reduction map $\phi \mapsto \phi_{\Delta_K}$ does not depend on $K$, because of the inductive construction of the reduction map (see Proposition \ref{prophod}) and the fact that $\phi \mapsto \phi_{\Delta_K}$ does not depend on $K$ when $\phi$ is any of the unnested atomic formulas listed above.
\end{proof}

\begin{Convention} \label{convent}
For the rest of the paper, fix once and for all a choice of a cross-section $s_K:\Gamma\to K^{\times}$, for each finite extension $K$ of $\Q_p$. For each such $K$, a cross-section does indeed exist and corresponds to a choice of a uniformizer for $\Oo_K$.
\end{Convention}

\bc \label{0topinter}
The asymptotic existential $L_t\cup P$-theory of $\{\F_p[t]/(t^n):n\in \N\}$ is Turing reducible to the asymptotic existential-universal $L_{\text{val},\times}$-theory of $\{K:[K:\Q_p]=e(K/\Q_p)<\infty\}$.
\ec 
\begin{proof}
For $K/\Q_p$ totally ramified of degree $n$, let $\Delta_K$ be the interpretation of the $L_t\cup P$-structure $\F_p[t]/(t^n)$ in the $L_{\text{val},\times}$-structure $K$, provided by Proposition \ref{intercor}. The reduction map of $\Delta_K$ does not depend on $K$ and will simply be denoted by $\phi\mapsto \phi_{\Delta}$. \\
\textbf{Claim: }For any existential $\phi \in \text{Sent}_{L_t\cup P}$, the sentence $\phi_{\Delta}$ is equivalent to an $\exists \forall$-sentence.
\begin{proof}
For sentences $\phi$ of the form $\exists x \psi(x)$, where $\psi(x)$ is a quantifier-free formula without negations, this follows from the fact that $\Delta_K$ is an $\exists \forall$-interpretation. We may therefore focus on formulas $\psi(x)$ of the form $f(x_1,...,x_m,t)\neq 0$ (resp. $f(x_1,...,x_m,t)\notin P$). Such a formula is logically equivalent to $\lnot f(x_1,...,x_m,y)= 0 \land y=t$ (resp. $\lnot f(x_1,...,x_m,y)\in P \land y=t$). Now $\lnot (f(x_1,...,x_m,y)=0)_{\Delta}$ (resp. $\lnot (f(x_1,...,x_m,y)\in P)_{\Delta}$) is universal and $(y=t)_{\Delta}$ is existential-universal. The conclusion follows.
\qedhere $_{\textit{Claim}}$ \end{proof}
For any $\phi \in \text{Sent}_{L_t\cup P}$ and any totally ramified extension $K/\Q_p$ of degree $n$, we have that $\F_p[t]/(t^n)\models \phi \iff K\models \phi_{\Delta}$. For any given $n\in \N$, there are finitely many totally ramified extensions $K/\Q_p$ of degree $n$ (Proposition 14 \cite{Lang}). It follows that 
$$\F_p[t]/(t^n)\models \phi \mbox{ for almost all } n\in \N \iff$$
$$ K\models \phi_{\Delta} \mbox{ for almost all }K\mbox{ with }[K:\Q_p]=e(K/\Q_p)<\infty$$
The conclusion follows from the Claim.
\end{proof}


\section{Truncations of $(\N;0,1,+,\mid_p)$}
\subsection{Motivation} \label{motiv}
\subsubsection{Pheidas' work} 
In Theorem 1 \cite{Pheid}, Pheidas showed that the Diophantine problem for $(\N;0,1,+,\mid_p)$ is undecidable. The proof goes by defining multiplication via a positive existential formula and using Matiyasevich's negative solution to Hilbert's tenth problem  \cite{Mat}. In Lemma 1$(c)$ \cite{Pheid}, it is shown that the Diophantine problem for $(\N;0,1,+,\mid_p)$ can be encoded in the existential theory of $\F_p(\!(t)\!)$ in $L_t$ with a predicate $P$ for $\{0,1,t,t^2,...\}$. It follows that the latter is also undecidable.
\subsubsection{Plan of action} \label{planofaction}
Our intuition that $\F_p[t]/(t^n)$ approximates $\F_p[\![t]\!]$ when $n$ is large, suggests that we can adapt Pheidas' strategy and show that the asymptotic existential theory of $\{\F_p[t]/(t^n):n\in \N\}$ is also undecidable in $L_t\cup P$. This will be established in Proposition \ref{asymprings} and in combination with Corollary \ref{0topinter} will pave the way for proving Theorem \ref{main} in Section \ref{mainsec}.
\subsection{$p$-divisibility}
The notion of $p$-divisibility was introduced by Denef in \cite{Denef}, in order to show that the Diophantine problem of a polynomial ring of positive characteristic is undecidable. Given $n,m\in \N$ and $p\in \mathbb{P}$, we write $n\mid_p m$ if $m=p^s n$, for some $s\in \N$. Let us also write $L_{p-div}=\{0,1,+,\mid_p\}$ for the language of addition and $p$-divisibility.
\subsubsection{Encoding $\mid_p$ in $\F_p[t]/(t^n)$} 
The following result is due to Pheidas:
\bl [Lemma 1$(a)$ \cite{Pheid}] \label{pheid}
Let $n,m\in \N$ with $0<n\leq m$. Then $n\mid_pm$ if and only if there exists $a \in \F_p[\![t]\!]$ such that $t^{-m}-t^{-n}=a^{-p}-a^{-1}$. 
\el 

\begin{rem} \label{pheidrem}
If $n \mid_p m$, then the proof of Lemma 1$(a)$ \cite{Pheid} provides $a=(t^{-n\cdot p^{s-1}}+t^{-n\cdot p^{s-2}}...+t^{-n})^{-1}$. Note that we have slightly rephrased the original formulation of Lemma 1$(a)$ \cite{Pheid}, so that the witness $a$ has positive valuation. 
\end{rem}
%
%
We shall use a truncated version of Lemma \ref{pheid}, whose proof is identical, modulo some additional bookkeeping:
\bl  \label{complemma}
Let $n,m,N \in\N$ with $0<n\leq m<N/3$. Then $n\mid_p m$ if and only if there exists $\alpha \in \F_p[t]/(t^N)$ such that $\alpha^p(t^n-t^m)=t^nt^m(1-\alpha^{p-1})$ and $\alpha^{3p} \neq 0$ in $\F_p[t]/(t^N)$.
\el 
\begin{proof}
$"\Rightarrow"$: Let $a =(t^{-n\cdot p^{s-1}}+t^{-n\cdot p^{s-2}}...+t^{-n})^{-1} \in \F_p[\![t]\!]$. After clearing denominators in Lemma \ref{pheid} , we get that $a^p(t^n-t^m)=t^nt^m(1-a^{p-1})$. Reducing the equation modulo $t^N$, yields $\alpha^p(t^n-t^m)=t^nt^m(1-\alpha^{p-1})$, where $\alpha$ is the image of $a$ in $\F_p[\![t]\!]/(t^N)\cong \F_p[t]/(t^N)$. Note that $v_t a^{3p}= 3m < N$ and thus $\alpha^{3p}\neq 0$ in $\F_p[t]/(t^N)$. \\
$"\Leftarrow"$: Let $n=p^rk$ and $m=p^li$ with $p\nmid k,i$. Let $\alpha$ be as in our assumption and $a\in \F_p[\![t]\!]$ be a lift of $\alpha$. Since $\alpha^{3p}\neq 0$, we get that $v_t a^p< N/3$. We will have by assumption that $t^{-m}-t^{-n}=a^{-p}-a^{-1}+t^{N-m-n}z a^{-p}$, for some $z \in \F_p[\![t]\!]$. Set $\varepsilon=t^{N-m-n}z a^{-p}$ and note that $v_t\varepsilon > 0$ because $m,n,v_t a^p< N/3$ and $v_t z\geq 0$.  By Lemma 1 \cite{Pheid}, we may find $a_1, a_2 \in \F_p[\![t]\!]$ such that $t^{-n}-t^{-k}=a_1^{-p}-a_1^{-1}$ and $t^{-m}-t^{-i}=a_2^{-p}-a_2^{-1}$. We compute that 
$$t^{-i}-t^{-k}=b^{p}-b+\varepsilon $$
where $b=a^{-1}+a_1^{-1}-a_2^{-1}$.

We claim that $i=k$. Otherwise, the left hand side must have negative valuation. Since $v_t \varepsilon > 0$, this forces the right hand side to have $p$-divisible valuation. This is contrary to the fact that $p\nmid i,k$. It follows that $i=k$ and thus $n\mid_p m$.
\end{proof}
\subsubsection{Interpreting $I_n$ in $\F_p[t]/(t^n)$} \label{truncatedversions}
Motivated by Lemma \ref{complemma}, we introduce  for each $n\in \N$ the $L_{p-div} \cup \{\infty\}$-structure $I_n=(\{0,1,...,n-1,\infty \};0,1,\infty,\oplus,\mid_p)$, where:
\begin{itemize}
 \item For $x,y\in I_n$, we have $x\mid_p y$ if $y=p^s x$ for some $s\in \N$ and $1\leq x,y< n/3$.\\
 \item  The operation $\oplus:I_n\times I_n\to I_n$ stands for truncated addition, i.e. given $x,y\in \{0,1,...,n-1 \}$ we have that $x\oplus y=x+y$ if $x+y<n$ and $\infty$ otherwise. Moreover, $\infty\oplus x=x\oplus \infty=\infty$ for all $x\in  \{0,1,...,n-1,\infty \}$.
\end{itemize}

\bp \label{inttorings1}
For each $n\in \N$, there is an $\exists$-interpretation $\Gamma_n$ of the $L_{p-div}\cup \{\infty\}$-structure $I_n$ in the $L_t\cup P$-structure $\F_p[t]/(t^n)$. Moreover, the reduction map $\phi\mapsto \phi_{\Gamma_n}$ does not depend on $n$.
\ep 
\begin{proof}

Take $\partial_{\Gamma_n}(x)$ to be the formula $x\in P$. The reduction map of $\Gamma_n$ on unnested atomic formulas is described as follows:
\begin{enumerate}
\item If $\phi$ is the formula $x=0$ (resp. $x=1$, $x=\infty$ and $x=y$), we take $\phi_{\Gamma_n} $ to be the formula $x=1$ (resp. $x=t$, $x=0$ and $x=y$). 

\item If $\phi(x,y,z)$ is $x+y=z$, then we take $\phi_{\Gamma_n}(x,y,z)$ to be the formula $x \cdot y=z$.

\item If $\phi(x,y)$ is the formula $x\mid_p y$, we take $\phi_{\Gamma_n}(x,y)$ to be the formula $\exists z (z^p(x-y)=x\cdot y(1-z^{p-1}) \land z^{3p}\neq 0)$.
\end{enumerate}

The coordinate map $f_{\Gamma_n}: \{0,1,t,...,t^{n-1}\} \to I_n$ is equal to $v|_{P}$, i.e., the valuation map $v$ restricted on $P$. Condition (3) of Definition \ref{interdef} is readily verified for unnested atomic formulas of type (1). For the formula described in (2), one has to use the isomorphism of monoids $(P, \cdot,1) \cong (\{0,1,...,n-1,\infty\},\oplus, 0)$. Using Lemma \ref{complemma}, one also verifies it for the unnested atomic formula of type (3). We deduce that $\Gamma_n$ is an interpretation. It is clear from the description of the reduction map on unnested formulas that $\Gamma_n$ is existential.

Finally, the reduction map $\phi \mapsto \phi_{\Gamma_n}$ does not depend on $n\in \N$, because of the inductive construction of the reduction map (see Proposition \ref{prophod}) and the fact that $\phi \mapsto \phi_{\Gamma_n}$ does not depend on $n$ when $\phi$ is any of the unnested atomic formulas listed above.

%
\end{proof}
\bc \label{inttorings}
The (resp. asymptotic) $L_{p-div}\cup \{\infty\}$-theory of $\{I_n:n\in \N\}$ is Turing reducible to the (resp. asymptotic) $L_t\cup P$-theory of $\{\F_p[t]/(t^n):n\in \N\}$.
\ec
\begin{proof}
Immediate from Proposition \ref{inttorings1}.
\end{proof}

\subsection{Undecidability of fragments} \label{undecoffragments}

Using the undecidability of the Diophantine problem over $(\N;0,1,+,\mid_p)$ (Theorem 1 \cite{Pheid}) as a black box, we shall prove that the asymptotic \textit{existential} theory of $\{I_n:n\in \N\}$ is also undecidable. 

The proof becomes more transparent by using matrix norms. Recall the \textit{maximum absolute row sum} norm $\|\cdot \|_{\infty}$ on the set of all matrices over $\Rr$, which is defined as $\| A \|_{\infty}= \max_{1\leq i \leq m} \sum_{j=1}^n |a_{ij} |$, for $A \in \text{M}_{m\times n} (\Rr)$. One readily checks that $\|\cdot \|_{\infty}$ is both sub-additive and sub-multiplicative, meaning that $\| A+B\|_{\infty}\leq \|A\|_{\infty}+\|B\|_{\infty}$ and $\| A\cdot B\|_{\infty}\leq \|A\|_{\infty}\cdot \|B\|_{\infty}$, whenever the operations are defined. Note also that when $A=(a_1,...,a_n)^{\top} \in \text{M}_{n\times 1} (\N)$, we have $\| A\|_{\infty}=\max \{a_i:i=1,...,n\}$.

For $A=(a_{ij})$ and $B=(b_{ij})$ in $ \text{M}_{m\times n} (\N)$, it will be convenient to use the notation $A\mid_p B$, which means that $a_{ij}\mid_p b_{ij}$ for each $1\leq i \leq m$ and $1\leq j \leq n$.
\bp \label{undecfragm}
The asymptotic existential $L_{p-div}\cup \{\infty\}$-theory of $\{I_n:n\in \N\}$ is undecidable. Moreover, let $T$ be the subtheory consisting of sentences $\phi $ of the form $\exists x \psi(x)$, where $x=(x_1,...,x_n)$ and $\psi(x)$ is a conjunction of a quantifier-free formula without negations with a formula of the form $\bigwedge_{i=1}^n N x_i \neq \infty$, for some $N\in \N$. Then $T$ is undecidable.
\ep
\begin{proof}
We shall encode the Diophantine problem of $(\N;0,1,+,\mid_p)$ in $T$. Let $\Sigma$ be an arbitrary system in variables $\mathbf{x}=(x_1,...,x_n)^{\top}\in \text{M}_{n\times 1} (\N)$ of the form

\begin{equation}
  \tag{$\Sigma$}
  \begin{cases}
 A_1 \mathbf{x}+\mathbf{b}_1=A_2 \mathbf{x}+\mathbf{b}_2 \\
  A_3 \mathbf{x}+\mathbf{b}_3 \mid_p A_4 \mathbf{x}+\mathbf{b}_4 
  \end{cases}
\end{equation}
where $A_i \in \text{M}_{m\times n} (\N)$ and $\mathbf{b}_i\in \text{M}_{m\times 1}(\N)$. Consider also the formula $\Sigma(x) \in \text{Form}_{L_{p-div}}$ associated with the system $\Sigma$.\\ 
\textbf{Claim:} We have that
$$(\N;0,1+,\mid_p)\models \exists x \Sigma(x)\iff T\models \exists x (\Sigma(x) \bigwedge_{i=1}^n 3 Mx_i\neq \infty )$$
where $M=\max_{1\leq i \leq 4} \{\|A_i \|_{\infty} + \|\mathbf{b}_i\|_{\infty}\}$.
\begin{proof}
 "$\Rightarrow$": Consider a witnessing tuple $\mathbf{c}=(c_1,...,c_n)^{\top} \in \text{M}_{n\times 1} (\N)$. If $\mathbf{c}= \mathbf{0}$, then the conclusion is clear. Otherwise, consider $m=\| \mathbf{c} \|_{\infty}=\max \{c_i:i=1,...,n\} \geq 1$ and choose $N\in \N$ such that $N>3\cdot M\cdot m$. 
 Using the sub-additive and sub-multiplicative properties of $\| \cdot \|_{\infty}$, we see that
 $$\| A_i  \mathbf{c}+\mathbf{b}_i \|_{\infty} \leq \|A_i\|_{\infty} \|\mathbf{c}\|_{\infty} +\|\mathbf{b}_i\|_{\infty}\leq m\cdot (\|A_i\|_{\infty} +\|\mathbf{b}_i\|_{\infty})\leq M\cdot m< N/3$$
for $i=1,...,4$. In particular, we get that both $\oplus$ and $\mid_p$ specialize to their ordinary counterparts in $\N$ and that $\Sigma(\mathbf{c})$ holds true in $I_N$, viewing $\mathbf{c}$ as a tuple in $I_N^n$. Each conjunct $3Mx_i\neq \infty$ also holds true for $c_i$ because $3Mc_i\leq 3M\cdot m<N$. \\
"$\Leftarrow$": Let $\mathbf{c}=(c_1,...,c_n)^{\top} \in I_N^n$ be a witness of the sentence $\exists x (\Sigma(x) \bigwedge_{i=1}^n 3Mx_i\neq \infty)$, for some $N>3\cdot M$. Since $3Mc_i\neq \infty$, we get that $c_i<N/3M$ for $i=1,...,n$. We therefore get that
$$ \| A_i \mathbf{c} +\mathbf{b}_i \|_{\infty} \leq \|A_i\|_{\infty} \|\mathbf{c}\|_{\infty} +\|\mathbf{b}_i\|_{\infty}< \frac{N}{3M} \cdot (\|A_i\|_{\infty} +\|\mathbf{b}_i\|_{\infty})\leq \frac{N}{3M}\cdot M= N/3 $$
for $i=1,...,4$.
In particular, we get that both $\oplus$ and $\mid_p$ specialize to their ordinary counterparts in $\N$. 
The corresponding tuple in $\N^n$ is the desired witness.  \qedhere $_{\textit{Claim}}$ \end{proof}
%
%
The conclusion follows from the fact that the Diophantine problem over $(\N;0,1,+,\mid_p)$ is undecidable (Theorem 1 \cite{Pheid}).
\end{proof}

\section{Proof of the main Theorem} 

\subsection{Truncated polynomial rings}
\bp \label{asymprings}
The asymptotic existential $L_t\cup P$-theory of $\{\F_p[t]/(t^n):n\in \N\}$ is undecidable.
\ep
\begin{proof}
Let $T$ be as in Proposition \ref{inttorings1}. We shall argue that the asymptotic existential $L_t\cup P$-theory of $\{\F_p[t]/(t^n):n\in \N\}$ is Turing reducible to $T$. The conclusion will then follow from Proposition \ref{undecfragm}. Let $\Gamma_n$ be the $\exists$-interpretation of the $L_{p-div}\cup \{\infty\}$-structure $I_n$ in the $L_t\cup P$-structure $\F_p[t]/(t^n)$, provided by Proposition \ref{inttorings1}. The reduction map of $\Gamma_n$ does not depend on $n\in \N$ and will simply be denoted by $\phi \mapsto \phi_{\Gamma}$. Since $\Gamma_n$ is existential, whenever $\psi(x)$ is a quantifier-free $L_{p-div}\cup \{\infty\}$-formula, the formula $\psi_{\Gamma}(x)$ is an existential formula. Moreover, from the description of the reduction map on unnested atomic formulas (see the proof of Proposition \ref{inttorings1}), we have $(\bigwedge_{i=1}^n N x_i \neq \infty)_{\Gamma}=\bigwedge_{i=1}^n  (N x_i \neq \infty)_{\Gamma}= \bigwedge_{i=1}^n x_i^N \neq 0$. It follows that $T$ is Turing reducible to the asymptotic existential $L_t\cup P$-theory of $\{\F_p[t]/(t^n):n\in \N\}$.
\end{proof}

\subsection{Totally ramified extensions}
\bp \label{totram}
The asymptotic $\exists \forall$-theory of $\{K:[K:\Q_p]=e(K/\Q_p)<\infty\}$ is undecidable in $L_{\text{val},\times}$.
\ep  
\begin{proof}
This follows from Corollary \ref{0topinter} and Proposition \ref{asymprings}. 
\end{proof}
\begin{rem}
$(a)$ The same proof applies verbatim to any \textit{infinite} collection of totally ramified extensions of $\Q_p$, e.g., $\{\Q_p(\zeta_{p^{n}}):n\in \N\}$ or $\{\Q_p(p^{1/n}):n\in \N\}$.\\
$(b)$ We do not know if the asymptotic existential theory of $\{K:[K:\Q_p]=e(K/\Q_p)<\infty\}$ is decidable or not in $L_{\text{val},\times}$. 
\end{rem}
\subsection{Finite extensions} \label{mainsec}
\begin{Theor} \label{mainagain}
The asymptotic $\exists \forall$-theory of $\{K:[K:\Q_p]<\infty\}$ is undecidable in $L_{\text{val},\times}$. 
\end{Theor}
\begin{proof}
Let $T$ be the theory in question and $T_{tot}$ be the asymptotic theory of $\{K:[K:\Q_p]=e(K/\Q_p)<\infty\}$ in $L_{\text{val}}$ with a cross-section. We shall encode $T_{tot}$ in $T$. Given an $\exists \forall$-sentence $\phi$, we see that
$$T_{tot}\models \phi \iff T\models (k=\F_p) \to \phi $$
The formal counterpart of $(k=\F_p) \to \phi$ is logically equivalent to an $\exists \forall$-sentence. It follows that the asymptotic existential-universal $L_{\text{val},\times}$-theory of $\{K:[K:\Q_p]<\infty\}$ is Turing reducible to the asymptotic existential-universal $L_{\text{val},\times}$-theory of $\{K:[K:\Q_p]=e(K/\Q_p)<\infty\}$. The conclusion follows from Proposition \ref{totram}.
\end{proof}
\section{Final remarks}

\subsection{$L_{\text{val}}$ vs $L_{\text{rings}}$} \label{fin} 
In view of \cite{DMC} and by possibly increasing the complexity, one may replace $L_{\text{val},\times}$ with the $1$-sorted language of rings $L_{\text{rings}}$ together with a unary predicate $P$ for the image of the cross-section in $K$. We need to use the following:
\begin{fact}  [Theorem 2 \cite{DMC}]  \label{DMCfact}
There is an $\exists \forall$-formula $\phi(x)$ in $L_{r}$ such that $\Oo_K=\phi(K)$ for any henselian valued field $K$ with finite or pseudo-finite residue field.
\end{fact}
\begin{rem}
In fact, Theorem 2 \cite{DMC} is about the existence of an \textit{existential} formula in $L_{\text{rings}}\cup P_2^{AS}$, where $P_2^{AS}(x)=\exists y(x=y^2+y)$. However, any existential sentence in $L_{\text{rings}}\cup P_2^{AS}$ is equivalent to an $\exists \forall$-sentence in $L_{\text{rings}}$.
\end{rem}
By Fact \ref{DMCfact}, the asymptotic theory of $\{K:[K:\Q_p]<\infty\}$ in $L_{\text{val}}$ with a cross-section can be encoded in the asymptotic theory of $\{K:[K:\Q_p]<\infty\}$ in $L_{\text{rings}}\cup P$. By Theorem \ref{main}, the latter is undecidable.
Our use of the cross-section/predicate formalism is essential and we do not know whether the (asymptotic) theory of $\{K:[K:\Q_p]<\infty\}$ is decidable in $L_{\text{rings}}$. 

\subsection{Residue rings and $\F_p(\!(t)\!)$}
If the (asymptotic) theory of $\{K:[K:\Q_p]<\infty\}$ is decidable in $L_{\text{rings}}$, then by Corollary \ref{0topinter} this would also yield a positive answer to the following question:
\bq \label{ques}
Is the (asymptotic) theory of $\{\F_p[t]/(t^n):n\in \N\}$ decidable in $L_{\text{rings}}$?
\eq 
\bob \label{obex}
The asymptotic \textit{existential} theory of $\{\F_p[t]/(t^n):n\in \N\}$ is decidable in $L_{\text{rings}}$.
\eob 
\begin{proof}
Let $\F_p[t^{1/\infty}]$ be the direct limit of the \textit{injective} system $\varinjlim \F_p[t^{1/n}]$, where $\phi_{nm}:\F_p[t^{1/n}]\hookrightarrow \F_p[t^{1/m}]$ is the natural inclusion map for $n\mid m$. We write $\F_p[t^{1/n}]/(t)$ and $\F_p[t^{1/\infty}]/(t)$ for the quotients modulo $t$. Note that $\phi_{nm}(t\cdot \F_p[t^{1/n}])\subseteq t\cdot \F_p[t^{1/m}]$ and therefore the induced maps $\overline{\phi_{nm}}:\F_p[t^{1/n}]\hookrightarrow \F_p[t^{1/m}]$ and $\overline{\phi_n}:\F_p[t^{1/n}]/(t)\hookrightarrow \F_p[t^{1/\infty}]/(t)$ are also injective. \\
\textbf{Claim 1: } The asymptotic existential $L_{\text{rings}}$-theory of $\{\F_p[t]/(t^n):n\in \N\}$ is equal to $\text{Th}_{\exists}(\F_p[t^{1/\infty}]/(t))$. 
\begin{proof}
For each $n\in \N$, we have a ring isomorphism 
$$\F_p[t^{1/n}]/(t)\cong \F_p[t,X]/(X^n-t,t)\cong \F_p[X]/(X^n)\cong \F_p[t]/(t^n)$$ 
Now if $\phi\in L_{\text{rings}}$ is existential, then $\F_p[t^{1/\infty}]/(t)\models \phi$ if and only if $\F_p[t^{1/n}]/(t)\models \phi$ for all sufficiently large $n$. The conclusion follows.
\qedhere $_{\textit{Claim 1}}$ \end{proof}
Finally, we prove:\\
\textbf{Claim 2: } $\text{Th}_{\exists}(\F_p[t^{1/\infty}]/(t))$ is decidable in $L_{\text{rings}}$.

\begin{proof}
A straightforward adaptation of Proposition 6.2.1 \cite{KK} shows that
$$  \F_p[t^{1/\infty}]/(t) \models \exists x \bigwedge_{1\leq i,j\leq n} (f_i(x)=0\land g_j(x)\neq 0) \iff $$ 
$$ \F_p[\![t]\!](t^{1/\infty})\models \exists x\bigwedge_{1\leq i,j\leq n}(v(f_i(x))> v(g_j(x)))$$
where $f_i(x),g_j(x)\in \F_p[x]$ are any multi-variable polynomials in $x=(x_1,...,x_m)$ for $i,j=1,...,n$. Finally, the henselian valued field $\F_p(\!(t)\!)(t^{1/\infty})$ is existentially decidable in $L_{\text{val}}$ by Corollary 7.5 \cite{AnscombeFehm}.
\qedhere $_{\textit{Claim 2}}$ \end{proof}

\end{proof}
\begin{rem}
$(a)$ Observation \ref{obex} should be contrasted with Proposition \ref{asymprings}, which shows that the asymptotic existential theory of $\{\F_p[t]/(t^n):n\in \N\}$ is undecidable in $L_t\cup P$.\\
$(b)$ The proof of Observation \ref{obex} does not go through for the language $L_t$, as the ring isomorphisms in the proof of Claim 1 do not respect $t$. We do not know if the asymptotic existential theory of $\{\F_p[t]/(t^n):n\in \N\}$ is decidable in $L_t$. \\
$(c)$ On the other hand, the asymptotic \textit{positive} existential theory of $\{\F_p[t]/(t^n):n\in \N\}$ in $L_t$ is equal to $\text{Th}_{\exists^+}\F_p[\![t]\!]$ and is decidable by an \textit{effective} version of Greenberg's approximation theorem (see Theorem 3.1 and Theorem 6.1 \cite{vddbeck}).
\end{rem}

\section*{Acknowledgements}
I would like to thank E. Hrushovski, who suggested the key idea of this paper, and J. Koenigsmann for careful readings and various suggestions. I also thank J. Derakhshan for helpful discussions and two anonymous referees for several fruitful comments.
\bibliographystyle{alpha}
\bibliography{references2}
\Addresses
\end{document}